\documentclass[12pt]{article}
\usepackage[numbers,sort&compress]{natbib}

\usepackage{threeparttable}
\usepackage[english]{babel}
\usepackage{diagbox}
\usepackage{t1enc}
\usepackage{amsthm,amsmath,amssymb}
\usepackage[mathscr]{eucal}
\usepackage{latexsym}
\usepackage{graphicx}
\usepackage{booktabs}
\usepackage{multirow}

\newtheorem{theorem}{Theorem}

\newtheorem{proposition}[theorem]{Proposition}
\newtheorem{lemma}[theorem]{Lemma}

\newtheorem{conjecture}[theorem]{Conjecture}

\newtheorem{definition}[theorem]{Definition}

\DeclareMathOperator{\rank}{rank}
\DeclareMathOperator{\diag}{diag}
\DeclareMathOperator{\lcm}{lcm}
\DeclareMathOperator{\RO}{RO}
\title{On the levels of rational regular orthogonal matrices for generalized cospectral graphs}
\author
{
Wei Wang\thanks{Corresponding author: wangwei.math@gmail.com} \quad\quad Jiaojiao Luo\quad\quad Li Wang\\
{\footnotesize School of Mathematics and Physics, Anhui Polytechnic University, Wuhu 241000, P. R. China}
\date{}
}
\begin{document}
 \maketitle
\begin{abstract} For an $n$-vertex graph $G$ with adjacency matrix $A$, the walk matrix $W(G)$ of $G$ is the matrix $[e,Ae,\ldots,A^{n-1}e]$, where  $e$ is the all-ones vector. Suppose that $W(G)$ is nonsingular and $p$ is an odd prime such that $W(G)$ has rank $n-1$ over the finite field $\mathbb{Z}/p\mathbb{Z}$. Let $H$ be a graph that is  generalized cospectral with $G$, and $Q$ be the corresponding rational regular orthogonal  matrix satisfying $Q^\mathsf{T} A(G) Q=A(H)$. We prove that 
	\begin{equation*}
		v_p(\ell(Q))\le \frac{1}{2}v_p (\det W(G))
	\end{equation*}
	where $\ell(Q)$ is the minimum positive integer $k$ such that $kQ$ is an integral matrix, and $v_p(m)$ is the maximum nonnegative integer $s$ such that $p^s$ divides $m$. This significantly  improves upon a recent result of Qiu et al. [Discrete Math. 346 (2023) 113177] stating that 
	$v_p(\ell(Q))\le v_p (\det W(G))-1.$
	
	\vspace{1em}


\end{abstract}
\section{Introduction}
For a graph $G$, the \emph{spectrum} of $G$ is the multiset of all eigenvalues of its adjacency matrix. Two graphs are \emph{cospectral} if they share the same spectrum.   We say two graphs $G$ and $H$ are \emph{generalized cospectral} if (i) $G$ is cospectral with $H$ and (ii) $\overline{G}$ is cospectral with $\overline{H}$, where $\overline{G}$ denotes the complement of $G$. A graph $G$ is \emph{determined by its generalized spectrum} (DGS)  if any graph generalized cospectral with $G$ is isomorphic to $G$. An orthogonal matrix $Q$ is \emph{regular} if the sum of each row is 1.  A classic result of Johnson and Newman \cite{johnson} states that two graphs $G$ and $H$ are generalized cospectral if and only if their adjacency matrices are similar via a regular orthogonal matrix. 

For an $n$-vertex graph $G$ with adjacency matrix $A$, the \emph{walk matrix} of $G$ is \begin{equation*}
W(G):=[e,Ae,\ldots,A^{n-1}e],
\end{equation*} where $e$ is the all-ones vector. A graph is \emph{controllable} \cite{godsil} if $W(G)$ is nonsingular. We remark that the set of controllable graphs is \emph{closed} under generalized cospectrality, that is, any graph generalized cospectral with a controllable graph must be controllable. A key observation of Wang \cite{wang2006} is that, when restricted to controllable graphs,   the regular orthogonal matrix connecting the adjacency matrices of two generalized cospectral graphs $G$ and $H$ is unique and rational. 

We are mainly concerned with controllable graphs in this paper. We use $\RO_n(\mathbb{Q})$ to denote the group of rational regular orthogonal matrices of order $n$. For a controllable graph $G$ with adjacency matrix $A$, we write
\begin{equation*}
	\mathcal{Q}(G)=\{Q\in \RO_n(\mathbb{Q})\colon\, Q^\mathsf{T} AQ \text{~is a $(0,1)$-matrix} \}.
	\end{equation*}
It is easy to show that any $(0,1)$-matrix of the form $Q^\mathsf{T} A Q$ must be the adjacency matrix of some graph (see \cite[Remark 1]{qiu2023}). Let $\mathcal{C}(G)$ be the set of all graphs that are generalized cospectral with $G$. As there exists a one-to-one correspondence between $\mathcal{C}(G)$ and $\mathcal{Q}(G)$, the problem of determining the set $\mathcal{C}(G)$ can naturally be  reduced to that of determining the set $\mathcal{Q}(G)$. For example, $\mathcal{Q}(G)$ always contains all permutation matrices of order $n$ which  correspond to graphs that are isomorphic to $G$.  Moreover, if $\mathcal{Q}(G)$ contains a non-permutation matrix, then  $\mathcal{C}(G)$ contains a graph that is not isomorphic to $G$, i.e., $G$ is not DGS.  The following notion introduced in \cite{wang2006} is crucial to investigate the set $\mathcal{Q}(G)$.
\begin{definition}[\cite{wang2006}] \normalfont
 Let $Q$ be a rational matrix. The \emph{level} of $Q$, denoted by $\ell(Q)$, is the
smallest positive integer $k$ such that $kQ$ is an integral matrix.
\end{definition}
Let 
\begin{equation*}
L(G)=\lcm\{\ell(Q)\colon\,Q\in \mathcal{Q}(G)\},
\end{equation*}the least common multiple of all levels  $\ell(Q)$ for $Q\in \mathcal{Q}(G)$. Clearly, if $L(G)=1$ then $\mathcal{Q}(G)$ contains only permutation matrices which implies that $G$ is DGS.   A theorem of Wang \cite{wang2013EJC,wang2017JCTB} states that if the integer $2^{-\lfloor\frac{n}{2}\rfloor}\det W(G)$ is  odd and square-free, then   $L(G)=1$ and hence $G$ is DGS. Wang proves this theorem by excluding the possibilities of $p\mid L(G)$ for odd primes $p$ and for $p=2$ separately.
\begin{proposition}[\cite{wang2013EJC}]\label{oddcase}
	If $p$ is an odd prime with $p^2\nmid \det W(G)$ then $p\nmid L(G)$.
\end{proposition}
\begin{proposition}[\cite{wang2017JCTB}]\label{evencase}
	If $2^{-\lfloor\frac{n}{2}\rfloor}$ is odd then $L(G)$ is odd.
\end{proposition}
Recently, Qiu et al.~\cite{qiu2023} improved both Proposition \ref{oddcase} and Proposition \ref{evencase} using a new and much  simpler argument. The current paper is a further improvement of the result in \cite{qiu2023}, but only for the odd prime case.  The result of  Qiu et al.~\cite{qiu2023} for the odd prime case is the following.
\begin{proposition}[\cite{qiu2023}]\label{qiuodd} If $p$ is an odd prime with $\rank_p W(G)=n-1$ then $L(G)\mid \frac{\det W(G)}{p}$.
\end{proposition}
For a prime $p$ and a nonzero integer $m$, we use $v_p(m)$ to denote the \emph{$p$-adic valuation} \cite{gouvea} of $m$, that is, the maximum nonnegative integer $k$ such that $p^k$ divides $m$. Note that the conclusion of Proposition  \ref{qiuodd} can be rewritten as the inequality $v_p(L(G))\le v_p(\det W(G))-1$. The main aim of this paper is to give a better upper bound of $v_p(L(G))$ under the same assumption of Proposition \ref{qiuodd}.
\begin{theorem}\label{main}
	Let $G$ be a controllable graph and $p$ be an odd prime such that  $\rank_p W(G)=n-1$. Then $v_p(L(G))\le \frac{1}{2} v_p(\det W(G))$.
\end{theorem}
	The main innovation of this paper lies in the shift of the underlying algebraic structure from the vector space over the field $\mathbb{Z}/p\mathbb{Z}$ (as used in \cite{wang2013EJC,qiu2023}) to the module over the local ring $\mathbb{Z}/p^k\mathbb{Z}$. While working over a ring introduces the complexity of zero divisors, many basic results from linear algebra over fields still hold for the ring $\mathbb{Z}/p^k\mathbb{Z}$. In particular, it possesses the ``Basis Extension Property'' (Steinitz Property), which allows us to generalize arguments from linear algebra. 
	
	The rest of this paper is organized as follows.  In Sec.~\ref{pre}, we recall some  preliminaries on Smith normal forms and modules over $\mathbb{Z}/p^k \mathbb{Z}$.  The proof of Theorem \ref{main} is given in Sec.~\ref{pf}, which can be seen as a module-theoretical version of the original argument of Wang \cite{wang2006}, combined with some improvements and simplifications developed in \cite{wang2020}.   Some direct applications of Theorem \ref{main} are presented in Sec.~\ref{dis}. The paper concludes with a conjecture proposing a further possible improvement to Theorem \ref{main}.
\section{Preliminaries}
\label{pre}
Let $R$ be a commutative principal ideal ring with identity.  It is well known \cite{stanley2016JCTA} that any matrix $M$ over $R$ has a Smith normal form (SNF); that is, there exist invertible matrices $U$ and $V$ over $R$ such that 
\begin{equation*}
	UMV=\begin{bmatrix}
		\diag(d_1,\ldots,d_r)&0\\
		0&0
	\end{bmatrix},
\end{equation*}
where the elements $d_i$ ($i=1,\ldots,r$) are nonzero and satisfy $d_1\mid d_2\mid \cdots\mid d_r$. The nonzero elements $d_1,\ldots,d_r$ are called the \emph{invariant factors} of $M$.

Throughout this paper, we fix an odd prime $p$ and consider three kinds of rings: $\mathbb{Z}$, $\mathbb{Z}/p\mathbb{Z}$ and $\mathbb{Z}/p^k\mathbb{Z}$, where $k\ge 2$. Let $\mathbb{Z} \to \mathbb{Z}/p^k \mathbb{Z} \to \mathbb{Z}/p\mathbb{Z}$ be the natural projections. For a matrix $M$ over $\mathbb{Z}$, the SNFs of the projections of $M$ are determined naturally by the SNF of $M$ over $\mathbb{Z}$ via $p$-adic valuations. For example, suppose $M$ has the Smith normal form $\diag(2,10,30,270)$ over $\mathbb{Z}$. Then, over the field $\mathbb{Z}/3\mathbb{Z}$, the SNF of $M$ is $\diag(1,1,0,0)$, whereas over the local ring $\mathbb{Z}/3^2\mathbb{Z}$, its SNF is $\diag(1,1,3,0)$. 

In this paper, we are mainly concerned with the SNFs of matrices over the local ring $\mathbb{Z}/p^k \mathbb{Z}$. Note that the invariant factors are unique only up to multiplication by units. To ensure uniqueness, we conventionally assume that each nonzero invariant factor takes the form $p^{c}$ for some integer $0 \le c < k$. We formalize this in the following lemma.

\begin{lemma}
	For any matrix $M$ over $\mathbb{Z}/p^k\mathbb{Z}$, there exist invertible matrices $U$ and $V$ over $\mathbb{Z}/p^k\mathbb{Z}$ such that 
	\begin{equation*}
		UMV = \begin{bmatrix}
			\diag(p^{c_1}, \ldots, p^{c_r})&0\\
			0&0
		\end{bmatrix},
	\end{equation*}
	where $0 \le c_1 \le \cdots \le c_r < k$.
\end{lemma}
A classic result in linear algebra states that the linear system $Mx=b$ over a field has a solution if and only if $\rank M = \rank (M,b)$. The corresponding question over an arbitrary ring is generally more complicated, but for our purposes, the existence of the SNF gives a straightforward extension.

\begin{proposition}\label{ssnf}
	The linear system $Mx=b$ over $\mathbb{Z}/p^k\mathbb{Z}$ has a solution if and only if $M$ and $(M,b)$ have the same invariant factors.
\end{proposition} 

\begin{proof}
	The ``only if'' part is clear. Let $\mathcal{C}(M)$ and $\mathcal{C}(M,b)$ be the modules generated by the columns of $M$ and $(M,b)$, respectively. Since $M$ and $(M,b)$ have the same invariant factors, the two modules $\mathcal{C}(M)$ and $\mathcal{C}(M,b)$ are isomorphic. As $\mathcal{C}(M,b)$ is finite, it cannot be isomorphic to any proper submodule. Since $\mathcal{C}(M)$ is a submodule of $\mathcal{C}(M,b)$, the isomorphism $\mathcal{C}(M)\cong \mathcal{C}(M,b)$ implies that they must be equal. Thus, $b\in\mathcal{C}(M)$, i.e., $Mx=b$ has a solution. 
\end{proof}
The SNF is particularly useful for determining the structure of the solution space (kernel) of a linear system. It is important to note that the kernel is not always a free module. However, under specific conditions on the invariant factors, it is.

\begin{proposition}\label{kf}
	Let $M$ be an $m \times n$ matrix over the ring $\mathbb{Z}/p^k\mathbb{Z}$. Let the invariant factors of $M$ be $p^{c_1}, \ldots, p^{c_r}$ satisfying $0 \le c_1 \le \dots \le c_r < k$. Then the kernel of $M$ is given, up to isomorphism, by:
	\begin{equation*}
		\ker(M) \cong \bigoplus_{i=1}^r \left( \mathbb{Z}/p^{c_i}\mathbb{Z} \right) \oplus \left( \mathbb{Z}/p^k \mathbb{Z} \right)^{n-r}.
	\end{equation*}
	In particular, if $c_1=\cdots=c_r=0$, then $\ker(M)\cong (\mathbb{Z}/p^k\mathbb{Z})^{n-r}$, which is a free module of rank $n-r$.
\end{proposition}

The following proposition is a direct generalization of \cite[Lemma 7]{wang2013EJC}, where the case $m=n,k=2$ was considered. The original proof in \cite{wang2013EJC} remains valid for the general setting and is therefore omitted here.

\begin{proposition}[\cite{wang2013EJC}]\label{dnk}
	Let $M$ be an $m\times n~(m\ge n)$ integral matrix whose SNF is 
	\begin{equation*}
		\begin{bmatrix}
		\diag(d_1,\ldots,d_n)\\
		0_{(m-n)\times n}
		\end{bmatrix}.
	\end{equation*} Then the equation $Mz\equiv 0\pmod{p^{k}}$ has a solution $z\not\equiv 0\pmod{p}$ if and only if $p^{k}\mid d_n.$
\end{proposition}

A set of vectors $v_1, \ldots, v_m \in (\mathbb{Z}/p^k\mathbb{Z})^n$ is called \emph{linearly independent} if $\sum_{i=1}^m c_i v_i = 0$ implies $c_i = 0$ for all $i$. It is known that $v_1,\ldots, v_m$ are linearly independent if and only if their projections in the vector space $(\mathbb{Z}/p\mathbb{Z})^n$ are linearly independent (over the field $\mathbb{Z}/p\mathbb{Z}$). A key property of the ring $\mathbb{Z}/p^k\mathbb{Z}$ is the so-called Basis Extension Property (or Steinitz Property \cite{mcdonald}), which states that any linearly independent vectors in a free $\mathbb{Z}/p^k\mathbb{Z}$-module $M$ can always be extended to a basis of $M$. Since we only require the finite-rank case, we formalize it in the following proposition.

\begin{proposition}\label{bas}
	Let $M$ be a free submodule of rank $m$ in $(\mathbb{Z}/p^k\mathbb{Z})^n$. Let $v_1, \ldots, v_k \in M$ ($k\le m$) be linearly independent. Then:
	
	\textup{(i)} If $k=m$, the set $\{v_1, \ldots, v_k\}$ constitutes a basis of $M$.
	
	\textup{(ii)}  If $k< m$, there exist $m-k$ vectors $v_{k+1}, \ldots, v_{m} \in M$ such that $\{v_1,\ldots,v_m\}$ constitutes a basis of $M$.
\end{proposition}
\section{Proof of Theorem \ref{main}}\label{pf}
We begin with the fundamental matrix characterization of generalized cospectrality for graphs.

\begin{lemma}[\cite{johnson,wang2006}]
	Let $G$ and $H$ be two graphs of the same order. Then $G$ and $H$ are generalized cospectral if and only if there exists a regular orthogonal matrix $Q$ such that $Q^\mathsf{T} A(G) Q=A(H)$. Moreover, if $G$ is controllable, then $Q^\mathsf{T}=W(H)(W(G))^{-1}$, and hence $Q$ is unique and rational.
\end{lemma}

In the following, let $G$ be an $n$-vertex controllable graph and $p$ be an odd prime factor of $\det W(G)$ such that $\rank_p W(G)=n-1$. Denote $\tau=v_p(L(G))$, i.e., $\tau=\max\{v_p (\ell(Q))\colon\, Q\in \mathcal{Q}(G)\}$. Note that if $\tau=0$, then Theorem \ref{main} clearly holds. Thus, we may assume that $\tau\ge 1$. For simplicity, the adjacency matrix $A(G)$ and the walk matrix $W(G)$ will be denoted by $A$ and $W$, respectively, when there is no confusion.

%\begin{lemma}[\cite{wang2006}]\label{basic}
%	We have $\tau\le v_p(d_n)$, and hence $\tau\le v_p(\det W)$, where $d_n$ is the $n$-th invariant factor of the integral matrix $W$.
%\end{lemma}

\begin{lemma}\label{fourcong}
	There exist an integral vector $z_0$ with $z_0\not\equiv 0\pmod{p}$ and an integer $\lambda_0$ such that $z_0^\mathsf{T} z_0\equiv 0\pmod{p^{2\tau}}$, $z_0^\mathsf{T} A z_0\equiv 0\pmod{p^{2\tau}}$, $W^\mathsf{T} z_0\equiv 0\pmod{p^\tau}$, and $Az_0\equiv \lambda_0z_0\pmod{p^\tau}$.
\end{lemma}

\begin{proof}
	Let $Q\in \mathcal{Q}(G)$ be such that $v_p(\ell(Q))=\tau$, and let $H$ be the corresponding graph. Let $\ell=\ell(Q)$ and $\hat{Q}=\ell Q$. Then $\hat{Q}$ is an integral matrix and $\hat{Q}\not\equiv 0\pmod{p}$. Let $z_0$ be a column of $\hat{Q}$ such that $z_0\not\equiv 0\pmod{p}$. Since $Q$ is an orthogonal matrix and $p^\tau\mid \ell$, we have $\hat{Q}^\mathsf{T} \hat{Q}=\ell^2 I$, and hence $z_0^\mathsf{T} z_0\equiv 0\pmod{p^{2\tau}}$. Similarly, as $Q^\mathsf{T} A Q$ is a $(0,1)$-matrix, we have $\hat{Q}^\mathsf{T} A\hat{Q}\equiv 0\pmod{p^{2\tau}}$, and hence $z_0^\mathsf{T} A z_0\equiv 0\pmod{p^{2\tau}}$. Moreover, noting that $W^\mathsf{T} Q$ equals $(W(H))^\mathsf{T}$, which is an integral matrix, we find that $W^\mathsf{T} \hat{Q}\equiv 0\pmod{\ell}$, and hence $W^\mathsf{T} z_0\equiv 0\pmod{p^\tau}$.
	
	Let $M=\{z\in \mathbb{Z}^n\colon\, W^\mathsf{T} z\equiv 0\pmod {p^\tau}\}$. Since $\rank_p W=n-1$ and $p^\tau \mid \det W$, we see that, over $\mathbb{Z}/p^\tau \mathbb{Z}$, the SNF of $W^\mathsf{T}$ is $\diag(1,1,\ldots,1,0)$. It follows from Proposition \ref{kf} that $M$ is a free $\mathbb{Z}/p^\tau \mathbb{Z}$-module of rank one. Since $z_0\in M$ and $z_0\not\equiv 0\pmod{p}$, we see that $\{z_0\}$ forms a basis for $M$ over $\mathbb{Z}/p^\tau \mathbb{Z}$ by Proposition \ref{bas} (i). On the other hand, by the Cayley-Hamilton Theorem, it is easy to see that $M$ is an $A$-invariant $\mathbb{Z}/p^\tau \mathbb{Z}$-submodule. This implies that there exists an integer $\lambda_0$ such that $Az_0\equiv \lambda_0 z_0\pmod{p^\tau}$. This completes the proof.	
\end{proof}
\begin{lemma}\label{wy}
	If $(A-\lambda_0 I)y\equiv s p^j z_0\pmod{p^{j+\tau}}$ for some integer $s$ and integer $j\ge 0$, then
	\begin{equation*}
		W^\mathsf{T} y\equiv (e^\mathsf{T} y)(1,\lambda_0,\ldots,\lambda_0^{n-1})^\mathsf{T} \pmod{p^{j+\tau}}.
	\end{equation*}
\end{lemma}

\begin{proof}
	We proceed by induction on $k$ to show that $e^\mathsf{T} A^k y\equiv \lambda_0^k e^\mathsf{T} y\pmod{p^{j+\tau}}$ for $k=0,1,\ldots,n-1$. The base case $k=0$ is trivial. Assume that the congruence holds for some integer $k$ with $0 \le k < n-1$; we proceed to verify it for $k+1$. 
	
	By Lemma \ref{fourcong}, we have $W^\mathsf{T} z_0\equiv 0\pmod{p^\tau}$, which implies that $e^\mathsf{T} A^k z_0\equiv 0\pmod{p^\tau}$. Consequently, we obtain $e^\mathsf{T} A^k (sp^j z_0)\equiv 0\pmod{p^{j+\tau}}$. By the hypothesis of this lemma, $Ay \equiv  sp^j z_0+\lambda_0 y \pmod{p^{j+\tau}}$, which implies that
	\[
	e^\mathsf{T} A^{k+1}y \equiv e^\mathsf{T} A^k (sp^j z_0+\lambda_0 y) \equiv \lambda_0 e^\mathsf{T} A^k y \equiv \lambda_0^{k+1}e^\mathsf{T} y \pmod{p^{j+\tau}}.
	\]
	This completes the proof.	
\end{proof}
\begin{lemma}\label{rA}
	Let $S=\diag(f_1,\ldots,f_n)$ be the SNF of $A-\lambda_0 I$. Then $f_{n-2}\not \equiv 0\pmod{p}$ and $f_n\equiv 0\pmod{p^\tau}$.
\end{lemma}

\begin{proof}
	By Lemma \ref{fourcong}, we have $(A-\lambda_0 I)z_0\equiv 0\pmod{p^\tau}$ and $z_0\not\equiv 0\pmod{p}$. It follows from Proposition \ref{dnk} that $f_n\equiv 0\pmod{p^\tau}$. 	It remains to show that $f_{n-2}\not \equiv 0\pmod{p}$. Suppose to the contrary that $f_{n-2}\equiv 0\pmod{p}$. Then $\rank_p (A-\lambda_0 I)\le n-3$. 
	Let \begin{equation*}
	B=\begin{bmatrix} A-\lambda_0 I\\e^\mathsf{T}\end{bmatrix}.\end{equation*}
	Then, we have $\rank_p B\le n-2$. Recall from Lemma \ref{fourcong} that $W^\mathsf{T} z_0 \equiv 0 \pmod{p}$, which implies $e^\mathsf{T} z_0 \equiv 0 \pmod{p}$. Together with $(A-\lambda_0 I)z_0 \equiv 0 \pmod{p}$, this yields $Bz_0 \equiv 0 \pmod{p}$. Since the nullity of $B$ over $\mathbb{Z}/p\mathbb{Z}$ is at least $n - (n-2) = 2$, the equation $Bz\equiv 0\pmod {p}$ must have a solution $z_1$ such that $z_0$ and $z_1$ are linearly independent over $\mathbb{Z}/p\mathbb{Z}$. 
	
	Noting that $Az_1\equiv \lambda_0 z_1 \pmod{p}$ and $e^\mathsf{T} z_1\equiv 0\pmod{p}$, we clearly have 
	\[
	e^\mathsf{T} A^{k}z_1\equiv \lambda_0^k e^\mathsf{T} z_1\equiv 0\pmod{p}
	\]
	for $k=0,1,\ldots, n-1$. This means that $W^\mathsf{T} z_1\equiv 0\pmod{p}$. By Lemma \ref{fourcong}, we also have $W^\mathsf{T} z_0\equiv 0\pmod{p}$. Therefore, $\rank_p W^\mathsf{T} \le n-2$. This contradicts the assumption that $\rank_p W = n-1$, which completes the proof.
\end{proof}
\begin{lemma} \label{lc}
		Let $M = [A-\lambda_0 I, z_0]$. Then $z_0^\mathsf{T} M \equiv 0 \pmod{p^\tau}$ and the SNF of $M$ is $[\diag(I_{n-1},0),0]$. Moreover, any integral  vector $z$ satisfying $z_0^\mathsf{T} z \equiv 0 \pmod{p^\tau}$ can be expressed as a linear combination of the columns of $M$ over $\mathbb{Z}/p^\tau \mathbb{Z}$.
	\end{lemma}
	\begin{proof} By Lemma \ref{fourcong}, we have $(A-\lambda_0 I)z_0\equiv 0\pmod{p^\tau}$ and $z_0^\mathsf{T} z_0\equiv 0\pmod{p^\tau}$. As  $A-\lambda_0 I$ is symmetric, we find that
		\begin{equation}\label{mz}
			M^\mathsf{T} z_0=\begin{bmatrix}
				A-\lambda_0 I\\
				z_0^\mathsf{T}
			\end{bmatrix}z_0\equiv 0\pmod{p^\tau}.
		\end{equation} Let the SNF of $M$ be $S= [\diag(p^{c_1},p^{c_2}, \dots, p^{c_n}), 0]$. Clearly, the SNF of $M^\mathsf{T}$ is $S^\mathsf{T}$. Note that $z_0\not\equiv 0\pmod{p}$. It follows from Eq.~\eqref{mz} and Proposition \ref{dnk}  that $p^{c_n}$ is zero over $R$. Thus, the SNF of $M$ can be simplified as 
		\begin{equation}\label{ss}
			S=[\diag(p^{c_1},\ldots,p^{c_{n-1}},0),0].
		\end{equation}

\noindent\textbf{Claim}: $p^{c_{n-1}}$ is a unit in $R$, i.e., $c_{n-1}= 0$.
		
		Suppose to the contrary that $c_{n-1}\ge 1$. Then, we have $\rank_p S\le n-2$, or equivalently, $\rank_p M\le n-2$. On the other hand, by Lemma \ref{rA}, we see that $\rank_p (A-\lambda_0 I)\ge n-2$ and hence  $\rank_p M \ge n-2$. Thus, we must have  $\rank_p M=\rank_p (A-\lambda_0 I)=n-2$. It follows that there exists an integral vector $z_1$ such that $(A-\lambda_0 I)z_1\equiv z_0\pmod{p}$.	
		
		As $\rank_p  (A-\lambda_0 I)=n-2$, $(A-\lambda_0I)z\equiv 0\pmod{p}$ has two  solutions $z_2$ and $z_3$ that are linearly  independent over $\mathbb{Z}/p\mathbb{Z}$. Since $(A-\lambda_0 I)z_1\equiv z_0\not \equiv 0\pmod{p}$, $z_1$ cannot be written as a linear combination of $z_2$ and $z_3$. This implies that $z_1,z_2,z_3$ are linearly independent. Consider the equation $e^\mathsf{T}(k_1z_1+k_2z_2+k_3z_3)\equiv 0 \pmod{p}$ with three unknowns $k_1,k_2,k_3$. Clearly, it has at least two independent solutions over $\mathbb{Z}/p\mathbb{Z}$. Let $(a_1,a_2,a_3)^\mathsf{T}$ and  $(b_1,b_2,b_3)^\mathsf{T}$  be  two such solutions and write $\alpha=a_1z_1+a_2z_2+a_3z_3$ and $\beta=b_1z_1+b_2z_2+b_3z_3$. It is easy to see that $\alpha$ and $\beta$ are linearly independent over $\mathbb{Z}/p\mathbb{Z}$. Note that $(A-\lambda_0I)\alpha\equiv a_1z_0$ and $e^\mathsf{T}\alpha\equiv 0\pmod{p}$. Using a similar argument as in the proof of  Lemma~\ref{wy}, we find  that $W^\mathsf{T}\alpha \equiv 0\pmod{p}$. Also,  $W^\mathsf{T} \beta\equiv 0\pmod{p}$. Thus, we have found two linearly independent solutions of $W^\mathsf{T} z\equiv 0\pmod{p}$. This contradicts the fact that $\rank_p W^\mathsf{T}=n-1$ and hence completes the proof of the Claim.
		
		By the Claim, we see that Eq.~\eqref{ss} can be further reduced to 
		\begin{equation}\label{s3}
			S=[\diag(I_{n-1},0),0].
		\end{equation}
		Let $\mathcal{C}(M)$ be the $R$-module generated by the columns of $M$ and $N$ be the module $\{z\in R^n\colon\, z_0^\mathsf{T} z=0\}$. By Eq.~\eqref{s3}, we know that $\mathcal{C}(M)$ is isomorphic to the free module $R^{n-1}$. Since $z_0\not\equiv 0\pmod{p}$, we find that $N$ is also isomorphic to $R^{n-1}$. Thus, the two $R$-modules $\mathcal{C}(M)$ and $N$ are isomorphic. Since $z_0^\mathsf{T} M=0$ over $R$, we know that $\mathcal{C}(M)$ is a submodule of $N$. As $R$ is a finite ring, we must have $\mathcal{C}(M)=N$, which completes the proof of Lemma \ref{lc}.
	\end{proof}
	
\begin{lemma}\label{ez}
	There exists an  integral vector $z_1$ such that $e^\mathsf{T} z_1\not\equiv 0\pmod {p}$ and $(A-\lambda_0 I)z_1\equiv p^cz_0\pmod{p^\tau}$ for some $c\in \{0,1,\ldots,\tau\}$.
\end{lemma}
\begin{proof}
	Let $R=\mathbb{Z}/p^\tau \mathbb{Z}$.  By Lemma \ref{rA}, we know that the SNF of $A-\lambda_0 I$ is
	\begin{equation}\label{smA}
		S=\diag(1,1,\ldots,1,p^c,0) \text{~for some~} c\in\{0,1,\ldots,\tau\}.
	\end{equation}
Let $M_i=[A-\lambda_0 I, p^i z_0]$ for $i\in\{0,1,\ldots,c\}$. 

\noindent\textbf{Claim 1}: The SNF of $M_i$ is $[\diag(1,1,\ldots,1,p^i,0),0]$ for $i\in\{0,1,\ldots,c\}$.

Let $S_i$ be the SNF of $M_i$. Noting that $\rank_p M_i\ge \rank_p (A-\lambda_0 I)\ge n-2$, we conclude that $S_i$ must have the form
\begin{equation}\label{si}
	S_i=[\diag(1,\ldots,1,p^{t_{i}},p^{t'_{i}}),0].
\end{equation}
Since $ M_i^\mathsf{T} z_0\equiv 0\pmod{p^\tau}$ and $z_0\not\equiv 0\pmod{p}$, Proposition \ref{dnk} indicates that $p^{t'_i}$ is the zero element in $R$. Thus, Eq.~\eqref{si} can be simplified as
\begin{equation}\label{st}
	S_i=[\diag(1,\ldots,1,p^{t_{i}},0),0].
\end{equation}
We need to show that $t_i=i$ for each $i\in\{0,1,\ldots,c\}$. Clearly, $t_0=0$ by Lemma \ref{lc}. Thus, it suffices to show that the sequence $\{t_i\}_{0\le i\le c}$ satisfies 
\begin{equation}\label{tc}
	t_c=c
\end{equation}
and 
\begin{equation}\label{icr}
	t_i\le t_{i+1}\le t_{i}+1\text{~for~} 0\le i<c.
\end{equation}
Now regard $A-\lambda_0 I$ and $M_i$ as matrices over $\mathbb{Z}$. Let $\Delta$ and $\Delta^{(i)}$ be the $(n-1)$-th determinantal factors of $A-\lambda_{0}I$ and $M_i$, respectively. Note that Eq.~\eqref{tc} holds for the case $c=\tau$. Indeed, if $c=\tau$ then $M_c=[A-\lambda_0 I,p^\tau z_0]=[A-\lambda_0 I,0]$ over $R$, which implies that the SNF of $M_c$ is $S_c=[S,0]$, i.e., $t_c=c$ as desired. Now assume $c<\tau$. From Eq.~\eqref{smA}, we know that $v_p(\Delta)=c$. Since $M_c$ contains $A-\lambda_0 I$ as a submatrix, we must have $v_p(\Delta^{(c)})\le v_p(\Delta)$ and hence  $v_p(\Delta^{(c)})\le c$. Let $D$ be any nonzero $(n-1)$-th minor of $M_c$. If $D$ does not include the last column of $M_c$, then $D$ is also a minor of $A-\lambda_0 I$ and hence $v_p(D)\ge v_p(\Delta)\ge c$. Otherwise,  each entry of the last column of $D$ is a multiple of $p^c$ and hence $v_p(D)\ge c$. Thus, we always have $v_p(D)\ge c$, which implies  $v_p(\Delta^{(c)})\ge c$. This means $v_p(\Delta^{(c)})= c$. But by Eq.~\eqref{st} for the case $i=c$, we know that $v_p(\Delta^{(c)})=t_c$ and hence $t_c=c$. This proves Eq.~\eqref{tc}.

We proceed to verify Eq.~\eqref{icr}. We may assume $c>0$ since otherwise we have nothing to show. Let $D^{(i)}$ and $D^{(i+1)}$ be any two corresponding $(n-1)$-th minors of $M_i$ and $M_{i+1}$, respectively. By the constructions of $M_i$ and $M_{i+1}$, we have either $D^{(i+1)}=D^{(i)}$ or $D^{(i+1)}=pD^{(i)}$. It follows that $v_p(\Delta^{(i)})\le v_p(\Delta^{(i+1)}) \le v_p(\Delta^{(i)}) +1$, i.e., $t_i\le t_{i+1}\le t_i+1$. Thus,  Claim 1 follows.

To complete the proof of Lemma \ref{ez}, we consider the following three cases:

\noindent\emph{Case 1}:  $c=0$. Since $A-\lambda_0 I$ and $M_0=[A-\lambda_0 I,z_0]$ have the same invariant factors over $R$, Proposition \ref{ssnf} implies that there exists an integral vector $z_1$ such that  \begin{equation}\label{as}
(A-\lambda_0 I)z_1\equiv z_0\pmod{p^\tau}.
\end{equation} By Lemma \ref{wy}, we have 
	\begin{equation*}
	W^\mathsf{T} z_1\equiv e^\mathsf{T} z_1(1,\lambda_0,\ldots,\lambda_0^{n-1})^\mathsf{T} \pmod{p^{\tau}}.
\end{equation*}
Suppose to the contrary that $e^\mathsf{T} z_1\equiv 0\pmod {p}$. Then we have $W^\mathsf{T} z_1\equiv 0\pmod{p}$.  Since $\rank_p W=n-1$ and $z_0$ is a nonzero (over $\mathbb{Z}/p\mathbb{Z}$) solution of $W^\mathsf{T} z\equiv 0\pmod{p}$, we conclude that $z_1\equiv kz_0\pmod{p}$ for some integer $k$. But this would imply 
\begin{equation*}
(A-\lambda_0 I) z_1\equiv k(A-\lambda_0 I) z_0\equiv 0\pmod{p},
\end{equation*}
which, combining with Eq.~\eqref{as}, leads to $z_0\equiv 0\pmod{p}$. This is a contradiction and hence we must have  $e^\mathsf{T} z_1\not\equiv 0\pmod {p}$.

\noindent\emph{Case 2}: $c=\tau$. As the SNF of $A-\lambda_0 I$ is $S=\diag(I_{n-1},0,0)$, we know that the kernel of $(A-\lambda_0 I)$ is a free $R$-module of rank $2$ by Proposition \ref{kf}. Let $K=\ker (A-\lambda_0 I)$ be the kernel. Since $z_0\in K$ and $z_0\not\equiv 0\pmod{p}$, Proposition \ref{bas} (ii) implies that there exists an integer vector $z_1$ such that $\{z_0,z_1\}$ (over $R$) constitutes a basis of the free module $K$. In particular, $(A-\lambda_0 I)z_1\equiv 0\pmod{p^\tau}$. Using Lemma \ref{wy}, we conclude that 
	\begin{equation*}
	W^\mathsf{T} z_1\equiv e^\mathsf{T} z_1(1,\lambda_0,\ldots,\lambda_0^{n-1})^\mathsf{T} \pmod{p^{\tau}}.
\end{equation*}
If $e^\mathsf{T} z_1\equiv 0\pmod{p}$ then we would have 	$W^\mathsf{T} z_1\equiv 0\pmod{p}$. But since $z_0$ and $z_1$ are independent over the field $\mathbb{Z}/p\mathbb{Z}$, we must have $\rank_p W^\mathsf{T}\le n-2$. This is a contradiction and hence $e^\mathsf{T} z_1\not\equiv 0\pmod{p}$.

\noindent{\emph{Case 3}}:  $0<c<\tau$. By Claim 1, we see that $A-\lambda_0 I$ and $M_c=[A-\lambda_0 I, p^c z_0]$ have the same invariant factors. By Proposition \ref{ssnf}, there exists an integral vector $z_1$ such that 
\begin{equation}\label{az1}
	(A-\lambda_0 I)z_1\equiv p^c z_0\pmod{p^\tau}.
\end{equation} 

\noindent\textbf{Claim 2}: $z_1$ and $z_0$ are linearly independent over $\mathbb{Z}/p\mathbb{Z}$.

Suppose to the contrary that $z_1$ and $z_0$ are linearly dependent over $\mathbb{Z}/p\mathbb{Z}$. As $z_0\not\equiv 0\pmod{p}$, there exist an integer $k$ and an integral vector $\tilde{z}_1$ such that $z_1=kz_0+p\tilde{z}_1$. Noting that $(A-\lambda_0 I)z_0\equiv 0\pmod{p^\tau}$, we obtain from Eq.~\eqref{az1} that $(A-\lambda_0 I)p\tilde{z}_1\equiv p^c z_0\pmod{p^\tau}$, i.e., 
\begin{equation}\label{az1s}
	(A-\lambda_0 I)\tilde{z}_1\equiv p^{c-1} z_0\pmod{p^{\tau-1}}.
	\end{equation}
Let $R'=\mathbb{Z}/p^{\tau-1}\mathbb{Z}$. By Eq.~\eqref{az1s}, we see that, over $R'$, the two matrices $A-\lambda_0 I$ and $M_{c-1}=[A-\lambda_0 I, p^{c-1}z_0]$ have the same invariant factors. Since $c\le \tau-1$, we see that $p^c$ and $p^{c-1}$ are not associates over $R'$. Thus, Claim 1 means that $A-\lambda_0 I$ and $M_{c-1}=[A-\lambda_0 I, p^{c-1}z_0]$ cannot have the same invariant factors. This contradiction completes the proof of Claim 2. 

By Eq.~\eqref{az1} and Lemma \ref{wy}, we have 
\begin{equation*}
	W^\mathsf{T} z_1\equiv e^\mathsf{T} z_1(1,\lambda_0,\ldots,\lambda_0^{n-1})^\mathsf{T}\pmod{p^\tau}.
\end{equation*}
By Claim 2 and the same argument as in Case 2, we must have $e^\mathsf{T} z_1\not\equiv 0\pmod{p}$. 

Combining the three cases,  Lemma \ref{ez} follows.
\end{proof}
Now we are in a position to present a proof of Theorem \ref{main}.

\begin{proof}[Proof of Theorem \ref{main}]  Let $\tau=v_p(L(G))$. Let $z_0$ and $\lambda_0$ be the integral vector and the integer as described in  Lemma \ref{fourcong}. Then we have 
\begin{equation*}
	z_0^\mathsf{T}(A-\lambda_0 I)z_0\equiv 0\pmod{p^{2\tau}}
\end{equation*}
and hence
\begin{equation*}
	z_0^\mathsf{T}\frac{(A-\lambda_0 I)z_0}{p^\tau}\equiv 0\pmod{p^{\tau}}.
\end{equation*}
It follows from Lemma \ref{lc} that 
\begin{equation}\label{aps}
\frac{(A-\lambda_0 I)z_0}{p^\tau}\equiv (A-\lambda_0 I)y+s z_0\pmod{p^{\tau}}
\end{equation}
for some $y\in \mathbb{Z}^{n}$ and $s\in \mathbb{Z}$.
Multiplying both sides of Eq.~\eqref{aps} by $p^\tau$ and rearranging the terms, we obtain
\begin{equation*}
	(A-\lambda_0 I)(z_0-p^\tau y)\equiv sp^\tau z_0\pmod{p^{2\tau}}.
\end{equation*}
Consequently, by Lemma \ref{wy}, we have
\begin{equation}\label{wp}
	W^\mathsf{T} (z_0-p^\tau y)\equiv e^\mathsf{T} (z_0-p^\tau y)(1,\lambda_0,\ldots,\lambda_0^{n-1})^\mathsf{T}\pmod{p^{2\tau}}.
\end{equation}
By Lemma \ref{ez}, there exist an integral vector $z_1$ and an integer $c\in\{0,1,\ldots,\tau\}$ such that $e^\mathsf{T} z_1\not\equiv 0\pmod{p}$ and
 \begin{equation}\label{apta}
(A-\lambda_0 I)z_1\equiv p^c z_0\pmod{p^\tau}.
\end{equation} As $e^\mathsf{T} (z_0-p^\tau y)\equiv e^\mathsf{T} z_0\equiv 0\pmod{p^\tau}$, we see that ${p^{-\tau}}e^\mathsf{T} (z_0-p^\tau y) $ is an integer. Note that $e^\mathsf{T} z_1$ is a unit in $\mathbb{Z}/p^{\tau}\mathbb{Z}$. Thus, there exists some integer $g$ such that 
\begin{equation}\label{ep}
	\frac{e^\mathsf{T} (z_0-p^\tau y)}{p^\tau}\equiv g e^\mathsf{T} z_1\pmod{p^{\tau}}.
\end{equation}
We know from Eq.~\eqref{apta} and Lemma \ref{wy} that 
\begin{equation*}
	W^\mathsf{T} z_1\equiv e^\mathsf{T} z_1(1,\lambda_0,\ldots,\lambda_0^{n-1})^\mathsf{T} \pmod{p^\tau}.
\end{equation*}
This, together with Eq.~\eqref{ep}, leads to 
\begin{equation*}
	W^\mathsf{T} (g z_1)\equiv g e^\mathsf{T} z_1(1,\lambda_0,\ldots,\lambda_0^{n-1})^\mathsf{T}\equiv 	\frac{e^\mathsf{T} (z_0-p^\tau y)}{p^\tau}(1,\lambda_0,\ldots,\lambda_0^{n-1})^\mathsf{T} \pmod{p^\tau}.
\end{equation*}
Multiplying both sides by $p^\tau$ and using Eq.~\eqref{wp}, we find that 
\begin{equation*}
W^\mathsf{T} (gp^\tau z_1) \equiv W^\mathsf{T} (z_0-p^\tau y)\pmod{p^{2\tau}},
\end{equation*}
i.e., 
\begin{equation*}
W^\mathsf{T} (z_0-p^\tau y-gp^\tau z_1) \equiv 0\pmod{p^{2\tau}}.
\end{equation*}
As $z_0-p^\tau y-g p^\tau z_1\equiv z_0\not\equiv 0\pmod{p}$, Proposition \ref{dnk} implies that $d_n\equiv 0\pmod{p^{2\tau}}$, where $d_n$ is the $n$-th invariant factor of the integral matrix $W^\mathsf{T}$. As $\rank_p W=n-1$, we have $v_p(\det W) =v_p(d_n)$ and hence $v_p(\det W)\ge 2\tau$. This completes the proof of Theorem \ref{main}.
\end{proof}
\section{Applications and future work}\label{dis}
As a natural extension of the arithmetic criterion of DGS-graphs established in \cite{wang2017JCTB}, Wang et al. \cite{wang2023ejc} defined, for each odd prime $p$,  a family of graphs 
 \begin{equation*}
\mathcal{F}_{n,p}=\{\text{$n$-vertex graphs~} G\colon\, 2^{-\lfloor n/2 \rfloor}\det W =p^2b\text{~and~} \rank_p W=n-1\},
 \end{equation*}
where $b$ is odd, square-free and $p\nmid b$. It was shown that each graph in $\mathcal{F}_{n,p}$ has at most one generalized cospectral mate \cite{wang2023ejc}. Using Theorem \ref{main}, we can show the same conclusion for a larger family of graphs:
 \begin{equation*}
 \widetilde{\mathcal{F}}_{n,p}=\{\text{$n$-vertex graphs~} G\colon\, 2^{-\lfloor n/2 \rfloor}\det W \in\{p^2b, p^3b\} \text{~and~} \rank_p W=n-1\},
 \end{equation*}
 where $p$ and $b$ satisfy the same restrictions. Indeed, let $G\in \widetilde{\mathcal{F}}_{n,p}$ and  $Q$ be any matrix in $\mathcal{Q}(G)$. Then, by Propositions \ref{oddcase} and \ref{evencase}, the only possible prime factor of $\ell(Q)$ is $p$. It follows from Theorem \ref{main} that
 \begin{equation*}
v_p(\ell(Q))\le \frac{1}{2}v_p(\det W)\le 3/2,
 \end{equation*}
 i.e., $v_p(\ell(Q))\in\{0,1\}$. In other words, $Q$ is either a permutation matrix, or a matrix with level exactly $p$. Now, the original argument in \cite{wang2023ejc} shows that $G$ has at most one generalized cospectral mate; refer to \cite{wang2023ejc} for details.
 
 The following small example illustrates the optimality of the family $\widetilde{\mathcal{F}}_{n,p}$ in the sense that $G$ may have two or more generalized cospectral mates if the $v_p(\det W)\ge 4$ for some odd prime $p$.
 
 \noindent\textbf{Example 1}. Let $n=10$ and $G$ be the $n$-vertex graph with adjacency matrix 
\begin{equation*}
	A={\footnotesize\begin{bmatrix}
		0 & 1 & 0 & 1 & 1 & 0 & 1 & 0 & 0 & 1 \\
	1 & 0 & 0 & 0 & 1 & 0 & 1 & 1 & 0 & 1 \\
	0 & 0 & 0 & 0 & 1 & 1 & 0 & 1 & 0 & 1 \\
	1 & 0 & 0 & 0 & 0 & 1 & 1 & 1 & 0 & 0 \\
	1 & 1 & 1 & 0 & 0 & 1 & 0 & 0 & 1 & 0 \\
	0 & 0 & 1 & 1 & 1 & 0 & 1 & 1 & 1 & 1 \\
	1 & 1 & 0 & 1 & 0 & 1 & 0 & 0 & 1 & 1 \\
	0 & 1 & 1 & 1 & 0 & 1 & 0 & 0 & 0 & 1 \\
	0 & 0 & 0 & 0 & 1 & 1 & 1 & 0 & 0 & 0 \\
	1 & 1 & 1 & 0 & 0 & 1 & 1 & 1 & 0 & 0 \\
	\end{bmatrix}}.
\end{equation*} 
It can be verified  that $2^{-\lfloor\frac{n}{2}\rfloor} \det W=3^4\times 19$ and $\rank_3 W=n-1$. By Theorem \ref{main}, we know that $v_3(L(G))\le \frac{1}{2} v_3(\det W)=2$. Indeed, using the program in \cite{wang2025}, we find that, besides the permutation matrices,  there exist essentially two rational regular orthogonal matrices 
\begin{equation*}
	Q_1=\frac{1}{3}{\footnotesize
	\begin{bmatrix}
		3 & 0 & 0 & 0 & 0 & 0 & 0 & 0 & 0 & 0 \\
		0 & 3 & 0 & 0 & 0 & 0 & 0 & 0 & 0 & 0 \\
		0 & 0 & 0 & 0 & -2 & 1 & 1 & 1 & 1 & 1 \\
		0 & 0 & 0 & 0 & 1 & -2 & 1 & 1 & 1 & 1 \\
		0 & 0 & 0 & 0 & 1 & 1 & -2 & 1 & 1 & 1 \\
		0 & 0 & 3 & 0 & 0 & 0 & 0 & 0 & 0 & 0 \\
		0 & 0 & 0 & 0 & 1 & 1 & 1 & -2 & 1 & 1 \\
		0 & 0 & 0 & 0 & 1 & 1 & 1 & 1 & -2 & 1 \\
		0 & 0 & 0 & 0 & 1 & 1 & 1 & 1 & 1 & -2 \\
		0 & 0 & 0 & 3 & 0 & 0 & 0 & 0 & 0 & 0 \\
	\end{bmatrix}}
\end{equation*}
and 
\begin{equation*}
Q_2=\frac{1}{9}{\footnotesize\begin{bmatrix}
		-3 & -3 & 3 & 3 & 3 & 0 & 0 & 0 & 0 & 6 \\
		3 & 3 & 6 & -3 & -3 & 0 & 0 & 0 & 0 & 3 \\
		1 & 1 & 2 & 2 & 2 & -6 & 3 & 3 & 3 & -2 \\
		1 & 1 & 2 & 2 & 2 & 3 & -6 & 3 & 3 & -2 \\
		-5 & 4 & -1 & -1 & -1 & 3 & 3 & 3 & 3 & 1 \\
		3 & 3 & -3 & 6 & -3 & 0 & 0 & 0 & 0 & 3 \\
		1 & 1 & 2 & 2 & 2 & 3 & 3 & -6 & 3 & -2 \\
		1 & 1 & 2 & 2 & 2 & 3 & 3 & 3 & -6 & -2 \\
		4 & -5 & -1 & -1 & -1 & 3 & 3 & 3 & 3 & 1 \\
		3 & 3 & -3 & -3 & 6 & 0 & 0 & 0 & 0 & 3 \\
\end{bmatrix}}
\end{equation*}
such that both $Q_1^\mathsf{T} A Q_1$ and $Q_2^\mathsf{T} A Q_2$ are adjacency matrices. Thus, $G$ has two generalized cospectral mates.

Recently, as a significant improvement of the criterion for a graph to have at most one cospectral mate established in \cite{wang2023ejc}, Raza et al. \cite{raza2026} obtained the following upper bound on the number of generalized cospectral mates of a graph.
\begin{theorem} [\cite{raza2026}]\label{ngm}
	Let $G$ be a controllable graph and $d_1,\ldots, d_n$ be the invariant factors of the walk matrix $W(G)$. Suppose $d_{\lceil\frac{n}{2}\rceil}=1$ and $d_{n-1}=2$. Let $d_n=\prod_{i=1}^k p_i^{m_i}$ be the prime factorization of $d_n$. Then $G$ has at most $\left(\prod_{i=1}^{k}m_i\right)-1$ generalized cospectral mates.
	\end{theorem}
Using the result of Qiu et al. \cite{qiu2023}, it is known that, for graphs $G$ as described in Theorem \ref{ngm},  the number $L(G)$ must be a factor of $\prod_{i=1}^k p_i^{m_i-1}$. Thus, $\ell(Q)$ can take at most  $\prod_{i=1}^k m_i$ different values (including 1) when $Q$ runs through $\mathcal{Q}(G)$. A key contribution of Raza et al.~\cite{raza2026} is to show that, under the assumptions of Theorem \ref{ngm}, there exists at most one matrix $Q$ in $\mathcal{Q}(G)$ for each possible level $\ell$ (in the sense of column permutations).  Noting that the matrix $Q\in \mathcal{Q}(G)$ with level $\ell=1$ is the permutation matrix, the number of generalized cospectral mates of $G$ is upper bounded by the number of possible levels, excluding 1. Thus, Theorem \ref{ngm} holds.

Using Theorem \ref{main} for odd primes, we can obtain a better bound on $L(G)$. Indeed, let $d_n=2^{m_1}\prod_{i=2}^kp_i^{m_i}$ be the factorization of $d_n$, where $p_2,\ldots,p_k$ are distinct odd primes. Then for each odd prime $p_i$, we have $v_{p_i}(L(G)) \le \lfloor\frac{1}{2}{m_i}\rfloor$ by Theorem \ref{main}. For $p=2$, we only have $v_2(L(G))\le m_1-1$ by the result of Qiu et al.~\cite{qiu2023}. Thus, the number of possible levels for matrices in $\mathcal{Q}(G)$, other than 1, is at most $m_1\left(\prod_{i=2}^k\left( \lfloor\frac{1}{2}m_i\rfloor+1\right)\right)-1$. This gives an improvement of Theorem \ref{ngm} under the same assumptions. We state it as the following theorem.
\begin{theorem}
	Let $G$ be a controllable graph and $d_1,\ldots, d_n$ be the invariant factors of the walk matrix $W(G)$. Suppose $d_{\lceil\frac{n}{2}\rceil}=1$ and $d_{n-1}=2$. Let $d_n=2^{m_1}\prod_{i=2}^k p_i^{m_i}$ be the prime factorization of $d_n$, where $p_2,\ldots,p_k$ are distinct odd primes. Then $G$ has at most  $m_1\left(\prod_{i=2}^k\left( \lfloor\frac{1}{2}m_i\rfloor+1\right)\right)-1$ generalized cospectral mates.
\end{theorem}
We end this paper by proposing a conjecture that the technical assumption on $\rank_p W(G)$ in Theorem \ref{main} is unnecessary for the conclusion. Moreover, the upper bound in terms of the determinant may be refined using the last two invariant factors.
 \begin{conjecture}
 	Let $G$ be a controllable graph and $p$ be an odd prime factor of $\det W(G)$. Let $d_{n-1}$ and $d_n$ be the last two invariant factors of $W(G)$. Then $v_p(L(G))\le \frac{1}{2}( v_p(d_{n})+v_p(d_{n-1}))$, and in particular,  $v_p(L(G))\le \frac{1}{2} v_p(\det W(G))$.
 \end{conjecture}

\section*{Acknowledgments}
This work is supported by the National Natural Science Foundation of China (Grant Nos. 12001006 and 12301042) and  Wuhu Science and Technology Project, China (Grant No. 2024kj015).


\begin{thebibliography}{99}
	
	\bibitem{godsil} 
	C.~D.~Godsil, Controllable subsets in graphs, Ann. Combin. 16 (2012) 733--744.
	
	
	\bibitem{gouvea} F.~Q.~Gouv\^{e}a, $p$-adic Numbers, 3rd ed., Springer, 2020.
	
	
	\bibitem{johnson} 
	C.~R.~Johnson, M.~Newman, A note on cospectral graphs, J. Combin. Theory, Ser. B 28 (1980) 96--103.
	
	
	\bibitem{mcdonald}
	B.~R.~McDonald, Linear Algebra Over Commutative Rings, Marcel Dekker, New York, 1984.

	\bibitem{qiu2023} 
	L.~Qiu, W.~Wang, W.~Wang, H.~Zhang, Smith Normal Form and the generalized spectral characterization of graphs, Discrete Math. 346 (2023) 113177.
	
	\bibitem{raza2026}
	M.~Raza, O.~U.~Ahmad, M.~Shabbir, W.~Abbas, On the enumeration of generalized cospectral mates of graphs, arXiv:2601.07373v1.
	
	\bibitem{stanley2016JCTA} 
	R. P. Stanley, Smith normal form in combinatorics, J. Combin. Theory, Ser. A 144 (2016) 476--495.
	

	\bibitem{wang2006} 
	W. Wang, C.-X. Xu, A sufficient condition for a family of graphs being determined by their generalized spectra, European J. Combin. 27 (2006) 826--840.
	
	\bibitem{wang2013EJC} 
	W. Wang, Generalized spectral characterization revisited, Electron. J. Combin. 20 (4) (2013) \#P4.
	
	\bibitem{wang2017JCTB} 
	W. Wang, A simple arithmetic criterion for graphs being determined by their generalized spectra, J. Combin. Theory, Ser. B 122 (2017) 438--451.
	
	\bibitem{wang2020} W.~Wang, L.~Qiu, J.~Qian, W.~Wang, Generalized spectral
	characterization of mixed graphs, Electron. J. Combin. 27(4) (2020)   \#P4.55.
	
	\bibitem{wang2025}
	W.~Wang, W.~Wang, Haemers' Conjecture: An algorithmic perspective, Exp. Math. 34 (2) (2025) 147--161.
	
	\bibitem{wang2023ejc} 
	W.~Wang, W.~Wang, T.~Yu, Graphs with at most one generalized cospectral mate, Electron. J. Combin. 30 (1) (2023) \#P1.38.
	
\end{thebibliography}
\end{document}